\documentclass[english]{amsart}

\usepackage{esint}
\usepackage[svgnames]{xcolor} 
\usepackage[colorlinks,citecolor=red,pagebackref,hypertexnames=false,breaklinks]{hyperref}
\usepackage{pgf,tikz}
\usepackage{pdfsync}

\usepackage{dsfont}
\usepackage{url}
\usepackage[utf8]{inputenc}
\usepackage[T1]{fontenc}
\usepackage{lmodern}
\usepackage{babel}
\usepackage{mathtools}  
\usepackage{amssymb}
\usepackage{lipsum}
\usepackage{mathrsfs}
\usepackage{color}

\newtheorem{theorem}{Theorem}[section]
\newtheorem{proposition}{Proposition}[section]
\newtheorem{lemma}{Lemma}[section]

\numberwithin{equation}{section}

\title[]{Determination of separable perturbations of an unbounded potential in the two-dimensional Schr\"odinger equation
}

\author{Mourad Choulli}
\address{Universit\'{e} de Lorraine, 34 cours L\'{e}opold, 54052 Nancy cedex, France}
\email{mourad.choulli@univ-lorraine.fr}

\author{Hiroshi Takase}
\address{Department of Mathematics, Okayama University, 3-1-1 Tsushima-naka, Kita-ku, Okayama 700-8530, Japan}
\email{takase@math.okayama-u.ac.jp}

\thanks{This work was supported by JSPS KAKENHI Grant Numbers JP25K17280 and JP23KK0049.}

\date{}

\begin{document}
\begin{abstract}
We establish uniqueness and stability results for a class of perturbations of an unbounded potential in the two-dimensional Schrödinger equation, from the corresponding Dirichlet-to-Neumann map. We assume that the difference between the potentials has a separated product structure. Our proof relies on a specific Carleman inequality.
\end{abstract}



\maketitle

\tableofcontents

\section{Introduction}\label{s1}

Let $\Omega$ be a $C^{0,1}$-bounded domain of $\mathbb{R}^2$ with boundary $\Gamma$, and $\alpha>1$. Unless otherwise stated, all functions we consider are complex-valued.

For $V\in L^\alpha (\Omega,\mathbb{R})$, we define the operator $A_V:H_0^1(\Omega)\rightarrow H^{-1}(\Omega)$ by
\[
\langle A_Vu,v\rangle = \int_\Omega [\nabla u\cdot \nabla \overline{v}+Vu\overline{v}]dx,\quad u,v\in H^1_0(\Omega),
\]
where $\langle \cdot ,\cdot\rangle$ is the duality pairing between $H_0^1(\Omega)$ and $H^{-1}(\Omega)$. We proceed as in \cite[Subsection 1.3]{Ch2} to show that $A_V$ is well defined and its spectrum, denoted by $\sigma(A_V)$, is discrete. We set $L_\ast ^\alpha (\Omega):=\{V\in L^\alpha(\Omega,\mathbb{R}),\; 0\not\in \sigma(A_V)\}$.

For $V\in L_\ast ^\alpha(\Omega)$ and $f\in H^{\frac{1}{2}}(\Gamma)$, consider the BVP
\begin{equation}\label{bvp1}
(-\Delta +V)u=0\; \mbox{in}\; \Omega,\quad u_{|\Gamma}=f.
\end{equation}
In light of the fact that $H^1(\Omega)$ is continuously embedded in $L^{\frac{4\alpha}{2\alpha+1}}(\Omega)$, we verify that \cite[Lemma 1.1 and Theorem 1.2]{Ch2} still hold for $n=2$ provided we replace $V\in L^{\frac{n}{2}}(\Omega,\mathbb{R})$ by $V\in L^\alpha (\Omega,\mathbb{R})$. Consequently, \eqref{bvp1} admits a unique solution $u_V(f)\in H^1(\Omega)$ satisfying
\begin{equation}\label{bou1}
\|u_V(f)\|_{H^1(\Omega)}\le \mathbf{c}_0\left(1+\|V\|_{L^\alpha (\Omega)}\right)\|f\|_{H^{\frac{1}{2}}(\Gamma)},
\end{equation}
where $\mathbf{c}_0=\mathbf{c}_0(\Omega,\mathrm{dist}(0,\sigma(A_V))) >0$ is a constant.

For $V\in L^\alpha(\Omega,\mathbb{R})$, let
\[
\mathscr{S}_V:=\{ u\in H^1(\Omega );\; (-\Delta +V)u=0\; \mbox{in}\; \Omega\}.
\]

The following lemma gives the definition of the normal derivative for functions in $\mathscr{S}_V$. It extends the usual normal derivative for functions in $H^2(\Omega)$.

\begin{lemma}\label{lemma-nor}
Let $V\in L^\alpha(\Omega,\mathbb{R})$, $u\in \mathscr{S}_V$ and define 
\[
\partial_\nu u : f\in H^{\frac{1}{2}}(\Gamma )\mapsto \partial_\nu u(f):=\int_\Omega \left(\nabla u\cdot \nabla \overline{F}+Vu \overline{F}\right)dx,
\]
where $F\in H^1(\Omega )$ is arbitrary such that $F_{|\Gamma}=f$. Then $\partial_\nu u$ is well defined, belongs to $H^{-1/2}(\Gamma )$, and the following inequality holds:
\begin{equation}\label{p6}
\|\partial_\nu u\|_{H^{-1/2}(\Gamma )}\le \mathbf{c}_1 \left(1+\|V\|_{L^\alpha(\Omega)}\right) \|u\|_{H^1(\Omega )},
\end{equation}
where $\mathbf{c}_1=\mathbf{c}_1(\Omega) >0$ is a constant.
\end{lemma}
The proof of Lemma \ref{lemma-nor} is identical to that of \cite[Lemma 1.2]{Ch2} ; we therefore  omit it.

We can now define the Dirichlet-to-Neumann map associated with $V\in L_\ast^\alpha(\Omega)$ by
\[
\Sigma_V:f\in H^{\frac{1}{2}}(\Gamma)\mapsto \partial_\nu  u_V(f)\in H^{-\frac{1}{2}}(\Gamma).
\]
By combining \eqref{bou1} and \eqref{p6}, we obtain that $\Sigma_V\in \mathscr{B}(H^{\frac{1}{2}}(\Gamma),H^{-\frac{1}{2}}(\Gamma))$.

Let $2<q'<4$, $\mu >\frac{2}{4-q'}$ and set
\[
\eta_1=\eta_1(q'):=\frac{q'}{q'-2}\quad\text{and}\quad \eta_2=\eta_2(q',\mu):=\frac{\mu q'-2}{\mu (4-q')-2}.
\]
Note that $q$, the conjugate of $q'$, satisfies $\frac{4}{3}<q<2$. 

Let $Q=I\times I'$ be a fixed bounded rectangle such that $Q\Supset \Omega$, where $I=]a,b[$ and $I'=]a',b'[$. Fix $\varkappa>0$ and define
\[
\mathscr{V}_\mu=\left\{ V\in L^\alpha_\ast(\Omega); V\chi_\Omega\in L^{\eta_2}(I,L^{\eta_1}(I')) \right\},
\]
where $V\chi_\Omega$ represents the extension of $V$ by zero to $Q\setminus \Omega$.

In the following, the Fourier-Laplace transform of $w\chi_I$ or $w\chi_{I'}$ will be denoted by $\hat{w}$. Here, $w\chi_I$ and $w\chi_{I'}$ denote the extension of $w$ by zero to $\mathbb{R}\setminus I$ and $\mathbb{R}\setminus I'$, respectively.

Our goal is to establish the following uniqueness result.

\begin{theorem}\label{thmu}
Let $\beta':=\frac{\mu q'-2}{\mu(q'-2)}$ $(>2)$, $b:=\min(q',\beta')$ and
\[
 \frac{b}{b-2}<\sigma <\frac{2b}{b-2}.
\]
If $V_1,V_2\in \mathscr{V}_\mu$ satisfy $(V_1-V_2)\chi_\Omega=U^1\otimes U^2\in L^{\sigma}(I, \mathbb{R})\otimes L^{\sigma}(I',\mathbb{R})$ and $\Sigma_{V_1}=\Sigma_{V_2}$, then $V_1=V_2$.
\end{theorem}

It is worth noting that $U^1$ and $U^2$ in the statement of Theorem \ref{thmu} satisfy $\mathrm{supp}(U^1\otimes U^2)=\mathrm{supp}(U^1)\times \mathrm{supp}(U^2)\subset \overline{\Omega}$. 

Theorem \ref{thmu} says that, for a given $V\in \mathscr{V}_\mu$,  the  Dirichlet-to-Neumann map determines uniquely the potential within the set $V+\chi_\Omega L^{\sigma}(I, \mathbb{R})\otimes L^{\sigma}(I',\mathbb{R})$. 

Although the result of Theorem \ref{thmu} may appear limited, to the best of our knowledge, only two uniqueness results in the literature address the determination of an unbounded potential for the two-dimensional Schr\"odinger equation from Cauchy data. The first result is due to Bukhgeim \cite{Bu} for a potential in $L^p$ with $p>2$. A recent improvement was obtained by Bl$\mathring{\rm a}$sten, Tzou, and Wang \cite{BTW} for a potential in $L^p$ with $p>\frac{4}{3}$. Furthermore, we find few results in the literature concerning stability. In the present work, we complement the aforementioned uniqueness results and establish a stability inequality. In three dimensions or higher, Dos Santos, Kenig, and Salo \cite{DKS} prove a uniqueness result in the case of an admissible manifold, for the optimal class of unbounded potentials. The case of a bounded domain in three dimensions was studied by the first author in \cite{Ch}, where he established uniqueness and stability results using ideas similar to those in \cite{DKS}.

More is known concerning the determination of the conductivity in the two-dimensional case from the Cauchy data. In the isotropic case, the uniqueness result has been established by Nachman \cite{Na} for $W^{2,p}$, $p>2$, conductivities. This result was extended to $W^{1,p}$ conductivities with $p>2$ by Brown and Uhlmann \cite{BU}, and subsequently generalized for $L^\infty$ conductivities by Astala and P\"aiv\"arinta \cite{AP}. We mention a generic uniqueness result by Sun and Uhlmann \cite{SU} for $W^{3,\infty}$ conductivities, as well as a logarithmic stability inequality established by Sun \cite{Su} for $C^4$ conductivities, subject to certain additional conditions. The anisotropic case for $L^\infty$ conductivities has been proved by Astala, P\"aiv\"arinta and Lassas \cite{APL}. The case of partial Cauchy data, for both isotropic and anisotropic smooth conductivities, has been obtained by Imanuvilov, Uhlmann and Yamamoto in \cite{IUY1, IUY2,IUY3}. 

The proof of Theorem \ref{thmu} is given in Section \ref{s3}. Beforehand, we construct in Section \ref{s2} special solutions of the Schr\"odinger operator with an unbounded potential which will be used in the proof of Theorem \ref{thmu}. The special solutions we have constructed are usually called CGO (complex geometrical optic) solutions. Our construction relies on a specific Carleman inequality, proven in detail in Appendix \ref{appendixA}. Finally, in Section \ref{s4} we demonstrate a quantitative version of Theorem \ref{thmu}.

\section{CGO solutions}\label{s2}

\subsection{Carleman inequality}

Let $\ell'=b'-a'$. As $2<q'<4$ and $\mu>\frac{2}{4-q'}$, we verify that
\begin{equation}\label{theta}
0<\theta:=\frac{\mu(q'-2)}{\mu q'-2}<\frac{1}{2}.
\end{equation}

Let $1<\beta<\gamma$ be chosen such that
\begin{equation}\label{bg}
\frac{1}{\beta}-\frac{1}{\gamma}=1-2\theta=\frac{\mu(4-q')-2}{\mu q'-2}.
\end{equation}

The proof of Carleman's inequality, stated below, consists of an adaptation of that of \cite[Proposition 2.1]{DKS}. For the sake of completeness, we provide a detailed proof in Appendix \ref{appendixA}. This Carleman inequality plays a key role in constructing CGO solutions.

\begin{theorem}\label{thmci}
For all $|\tau| \ge 4$ and $u\in C_0^\infty (Q)$, we have
\[
\|u\|_{L^\gamma(I,L^{q'}(I'))}\le \mathbf{c}(1+\ell^{\frac{1}{\beta'}+\frac{1}{\gamma}})\|(e^{\tau t}\Delta e^{-\tau t})u\|_{L^\beta(I,L^q(I'))},
\]
where $\ell=b-a$ and $\mathbf{c}=\mathbf{c}(\ell',q',\mu)>0$ is a constant.
\end{theorem}

\subsection{Constructing CGO solutions}

Let $1<\beta<2$ ; in this case, we can choose $\gamma=\beta'$ in Theorem \ref{thmci} (see the proof below).  The elements of $I$ will be denoted by $x$ and those of $I'$ will be denoted by $x'$.

Let 
\[
\Lambda:=\left\{\tau \in \mathbb{R};\; |\tau|\ge 4,\; |\tau|\ne \frac{j\pi}{\ell'},\; j\in \mathbb{N}\right\}.
\]

Thereafter, $\mathbf{c}=\mathbf{c}(\ell,\ell',q, \mu,\beta)>0$ will denote a generic constant.

In light of Theorem \ref{thmci}, proceeding as in the proof of \cite[Proposition 4.1]{KSU}, we obtain the following result.

\begin{proposition}\label{proposition1}
For all $\tau\in \Lambda$, there exists $E_\tau :L^2(Q)\rightarrow H^2(Q)$ such that
\begin{align*}
&e^{\tau x}(-\Delta)e^{-\tau x}E_\tau f=f,\quad f\in L^2(Q),
\\
&E_\tau e^{\tau x}(-\Delta)e^{-\tau x} f=f,\quad f\in C_0^\infty (Q).
\end{align*}
Furthermore, the following inequalities hold for all $f\in L^2(Q)$:
\begin{align}
&\|E_\tau f\|_{H^s(Q)}\le \mathbf{c}|\tau|^{s-1}\|f\|_{L^2(Q)},\quad 0\le s\le 2,\label{est1}
\\
&\|E_\tau f\|_{L^{\beta'}(I,L^{q'}(I'))}\le \mathbf{c}\|f\|_{L^\beta(I,L^q(I'))}.\label{est2}
\end{align}
\end{proposition}

The estimate \eqref{est2} shows that $E_\tau$ extends to a bounded operator from $L^\beta(I,L^q(I'))$ to $L^{\beta'}(I,L^{q'}(I'))$. For the sake of simplicity, this extension will also be denoted by $E_\tau$ in the sequel.

Before stating our next result, we make some formal calculations. In what follows, we use repeatedly the following H\"older type inequality:
\begin{equation}\label{fh}
\|fg\|_{L^\ell}\le \|f\|_{L^{\ell_1}}\|g\|_{L^{\ell_2}},\quad \frac{1}{\ell_1}+\frac{1}{\ell_2}=\frac{1}{\ell}.
\end{equation}

Let $V_0\in L^{\ell_2}(I, L^{\ell_1}(I'))$. Applying \eqref{fh}, we have
\[
\|V_0(x,\cdot)E_\lambda V_1f(x,\cdot)\|_{L^2(I')}\le \|V_0(x,\cdot)\|_{L^{\ell_1}(I')}\|E_\lambda V_1f(x,\cdot)\|_{L^{q'}(I')}
\]

with $\frac{1}{\ell_1}+\frac{1}{q'}=\frac{1}{2}$. Using once again \eqref{fh} and \eqref{est2}, we obtain
\begin{align*}
\|V_0E_\lambda V_1f\|_{L^2(Q)}&\le \|V_0\|_{L^{\ell_2}(I,L^{\ell_1}(I'))}\|E_\lambda V_1f\|_{L^\gamma(I,L^{q'}(I'))}
\\
&\le \mathbf{c}\|V_0\|_{L^{\ell_2}(I,L^{\ell_1}(I'))}\|V_1f\|_{L^\beta(I,L^q(I'))},
\end{align*}
where $\frac{1}{\ell_2}+\frac{1}{\gamma}=\frac{1}{2}$.

On the other hand, we proceed as above to get
\[
\|V_1f\|_{L^\beta(I,L^q(I'))}\le \|V_1\|_{L^{k_2}(I,L^{k_1}(I'))}\|f\|_{L^2(Q)},
\]
where $\frac{1}{k_1}+\frac{1}{2}=\frac{1}{q}$ and $\frac{1}{k_2}+\frac{1}{2}=\frac{1}{\beta}$. Consequently, we have
\[
\|V_0E_\lambda V_1f\|_{L^2(Q)}\le \mathbf{c}\|V_0\|_{L^{\ell_2}(I,L^{\ell_1}(I'))}\|V_1\|_{L^{k_2}(I,L^{k_1}(I'))}\|f\|_{L^2(Q)}.
\]

Let $V=|V|e^{i\varphi}\in L^{\eta_2}(I,L^{\eta_1}(I'))$, $V_0=|V|^se^{i\varphi}$ and $V_1=|V|^{1-s}$, where $\varphi:=\arg(V)$ and $s\in (0,1)$ is to be determined. Then
\begin{align*}
&\|V_0\|_{L^{\ell_2}(I,L^{\ell_1}(I'))}=\|V\|_{L^{s\ell_2}(I,L^{s\ell_1}(I'))}^s,
\\
&\|V_1\|_{L^{k_2}(I,L^{k_1}(I'))}=\|V\|_{L^{(1-s)k_2}(I,L^{(1-s)k_1}(I'))}^{1-s}.
\end{align*}
These identities impose the following conditions:
\[
s\ell_1=\eta_1,\quad s\ell_2=\eta_2,\quad (1-s)k_1=\eta_1\quad \text{and} \quad (1-s)k_2=\eta_2.
\]
In particular, we must have
\[
\frac{\eta_1}{\ell_1}=\frac{\eta_2}{\ell_2}=s\quad\text{and}\quad \frac{\eta_1}{k_1}=\frac{\eta_2}{k_2}=1-s.
\]
Equivalently, the following conditions must hold:
\[
\eta_1\left(\frac{1}{2}-\frac{1}{q'}\right)=\eta_2\left(\frac{1}{2}-\frac{1}{\gamma}\right)=s
\quad\text{and}\quad
\eta_1\left(\frac{1}{q}-\frac{1}{2}\right)=\eta_2\left(\frac{1}{\beta}-\frac{1}{2}\right)=1-s,
\]
that is,
\begin{equation}\label{con1}
\eta_1\frac{q'-2}{2q'}=\eta_2\frac{\gamma-2}{2\gamma}=s
\quad\text{and}\quad \eta_1\frac{2-q}{2q}=\eta_2\frac{2-\beta}{2\beta}=1-s.
\end{equation}
The following formula will be useful in the following: if $1<r<2$ and $r'=\frac{r}{r-1}$ is its conjugate, then
\[
\frac{r'-2}{r'}=\frac{2-r}{r}.
\]
Let us choose  $\gamma=\beta'$, which stands for the conjugate of $1<\beta<2$. Then, using the relation
\[
\frac{1}{\beta}-\frac{1}{\gamma}=\frac{\mu(4-q')-2}{\mu q'-2},
\]
we obtain
\[
\beta'=\frac{\mu q'-2}{\mu( q'-2)}.
\]
As $\mu>\frac{2}{4-q'}$, we verify that $\beta'>2$. With this choice of $\gamma$, \eqref{con1} gives
\begin{equation}\label{idd}
s=\frac{1}{2},\quad \eta_1=\frac{q'}{q'-2}
\quad\text{and}\quad \eta_2=\frac{\mu q'-2}{\mu(4-q')-2}.
\end{equation}
Remark that $\eta_1=\eta_1(q')$ is a decreasing function from $\left(2,4\right)$ onto $(2,\infty)$ and, in the particular case where $\mu=\mu(q')=\frac{3}{4-q'}$, $\eta_2=\eta_2(q')=\frac{5q'-8}{4-q'}$ is an increasing function from $\left(2,4\right)$ onto $\left(1,\infty\right)$.

In light of the previous discussion, we can state the following result.

\begin{proposition}\label{proposition2}
Let $\eta_1$ and $\eta_2$ be given by \eqref{idd}, $V=|V|e^{i\varphi}\in L^{\eta_2}(I,L^{\eta_1}(I'))$, $V_0=|V|^{\frac{1}{2}}e^{i\varphi}$ and $V_1=|V|^{\frac{1}{2}}$. Then, $V_j\in L^{2\eta_2}(I,L^{2\eta_1}(I'))$,
\[
\|V_j\|_{L^{2\eta_2}(I,L^{2\eta_1}(I'))}\le \|V\|_{L^{\eta_2}(I,L^{\eta_1}(I'))}^{\frac{1}{2}},\quad j=0,1.
\]
Furthermore, for all $\tau \in \Lambda$, we have
\[
\|V_0E_\tau V_1f\|_{L^2(Q)}\le \mathbf{c}\|V\|_{L^{\eta_2}(I,L^{\eta_1}(I'))}\|f\|_{L^2(Q)}, \quad f\in L^2(Q).
\]
\end{proposition}

Let $0<\epsilon<1$ and $V\in L^{\eta_2}(I,L^{\eta_1}(I'))$. Under the assumptions and the notations of Proposition \ref{proposition2}, let
\begin{equation}\label{ww}
V_j^0:=V_j\chi_{\{|V_j|\le \rho\}}\quad\text{and}\quad V_j^1:=V_j\chi_{\{|V_j|> \rho\}},\quad \rho>0,\; j=0,1,
\end{equation}
By the dominated convergence theorem,  $\|V_j^1\|_{L^{2\eta_2}(I,L^{2\eta_1}(I'))}$ tends to $0$ when $\rho$ goes to $\infty$. Consequently, we find  $\rho=\rho(V,\epsilon) >0$ sufficiently large in such a way that $\|V_j^1\|_{L^{2\eta_2}(I,L^{2\eta_1}(I'))}\le \epsilon$. On the other hand, as $\|V_j^0\|_{L^\infty(Q)}\le \rho$, $j=0,1$, we obtain from \eqref{est1}
\begin{equation}\label{est3}
\|V_0^0E_\tau V_1^0f\|_{L^2(Q)}\le \mathbf{c}\rho^2|\tau|^{-1}\|f\|_{L^2(Q)}.
\end{equation}

From now on, $\mathbf{c}(V)=\mathbf{c}(\ell,\ell',q, \mu,\beta,V)>0$ will denote a generic constant. We obtain from Proposition \ref{proposition2} 
\begin{align*}
&\|V_0^0E_\tau V_1^1f\|_{L^2(Q)}+ \|V_0^1E_\tau V_1^0f\|_{L^2(Q)}\le \mathbf{c}(V)\epsilon\|f\|_{L^2(Q)},
\\
&\|V_0^1E_\tau V_1^1f\|_{L^2(Q)}\le \mathbf{c}(V)\epsilon^2\|f\|_{L^2(Q)}.
\end{align*}
Then, it follows from \eqref{est3} that there exists $\tau_\ast=\tau_\ast(\ell,\ell',q',\mu,\gamma,V,\epsilon)>0$ such that for all $\tau \in \Lambda$ satisfying $|\tau| \ge \tau_\ast$, we have
\begin{equation}\label{est4.0}
\|V_0E_\tau V_1f\|_{L^2(Q)}\le \mathbf{c}(V)\epsilon\|f\|_{L^2(Q)}.
\end{equation}
In other words, we proved the following lemma.
\begin{lemma}
Let $\eta_1$ and $\eta_2$ be given by \eqref{idd}, and $V=|V|e^{i\varphi}\in L^{\eta_2}(I,L^{\eta_1}(I'))$. If  $V_0=|V|^{\frac{1}{2}}e^{i\varphi}$ and $V_1=|V|^{\frac{1}{2}}$, then
\begin{equation}\label{est4}
\lim_{\tau \in \Lambda,\, |\tau| \rightarrow \infty} \|V_0E_\tau V_1\|_{\mathscr{B}(L^2(Q))}=0.
\end{equation}
\end{lemma}

Let $\xi \in \mathbb{C}$, $\tau \in \Lambda$ and $\zeta=\tau-\xi$. Define $\phi:Q\rightarrow \mathbb{C}$ by
\begin{equation}\label{phi}
\phi(x,x'):=e^{\zeta(-x+ix')},\quad (x,x')\in Q.
\end{equation}
We verify that $\Delta \phi=0$.

Let $V$ be a measurable function defined on $\Omega$. For simplicity, its extension by $0$ in $Q$ is denoted again by $V$. Formally, we seek $u=\phi +e^{-\tau x}v$ satisfying $(-\Delta +V)u=0$ in $Q$. That is, $v$ we must satisfy the equation
\[
0=e^{\tau x}V\phi+e^{\tau x}(-\Delta) e^{-\tau x}v+Vv.
\]
Let us suppose that $v=E_\tau V_1 z$. Then, the equation above yields
\[
0=e^{\tau x}V_1V_0\phi+V_1z+V_1V_0E_\tau V_1 z.
\]
Therefore, it suffices to solve the equation
\begin{equation}\label{eqq1}
0=e^{\tau x}V_0\phi+z+V_0E_\tau V_1 z.
\end{equation}

From \eqref{est4.0}, there exists $\tau_0=\tau_0(\ell,\ell',q',\mu,\gamma,V)>0$ such that if $\tau \in \Lambda$, $|\tau|\ge \tau_0$, then
\[
\|V_0E_\tau V_1\|_{\mathscr{B}(L^2(Q))}\le \frac{1}{2},
\]
and hence $1+V_0E_\tau V_1$ is invertible in $\mathscr{B}(L^2(Q))$. As $e^{\tau x}V_0\phi \in L^2(Q)$,  
\begin{equation}\label{zz}
z:=(1+V_0E_\tau V_1)^{-1}e^{\tau x}V_0\phi 
\end{equation}
is a solution of \eqref{eqq1}. Let $0<\epsilon <1$ and choose $V_1^0$ and $V_1^1$ as in \eqref{ww} with  $\rho=|\tau|^{\frac{1}{4}}$. 
In this case, $\|V_1^0\|_{L^\infty(Q)}\le |\tau|^{\frac{1}{4}}$ and $\|V_1^1\|_{L^{2\eta_2}(I,L^{2\eta_1}(I'))}\le \epsilon$ provided that $|\tau|$ is large enough. Under this choice, using \eqref{est1} and \eqref{est2}, we get
\begin{align*}
\|v\|_{L^2(Q)}&\le \|E_\lambda V_1^0z\|_{L^2(Q)}+ \|E_\lambda V_1^1z\|_{L^2(Q)}
\\
&\le \mathbf{c}\left(|\tau|^{-\frac{3}{4}}\|z\|_{L^2(Q)}+\|V_1^1z\|_{L^\beta(I,L^q(I'))}\right)
\\
&\le \mathbf{c}\left(|\tau|^{-\frac{3}{4}}\|z\|_{L^2(Q)}+\|V_1^1\|_{L^{2\eta_2}(I,L^{2\eta_1}(I'))}\|z\|_{L^2(Q)}\right)
\\
&\le \mathbf{c}\left(|\tau|^{-\frac{3}{4}}+\epsilon\right)\|z\|_{L^2(Q)}.
\end{align*}
Since
\[
\|z\|_{L^2(Q)}\le 2\|V_0\|_{L^\infty(Q)} \|e^{\xi x+i\zeta x'}\|_{L^2(Q)},
\]
we obtain 
\begin{equation}\label{lim0}
\lim_{\tau \in \Lambda,\, |\tau| \rightarrow \infty}\|v\|_{L^2(Q)}=0.
\end{equation}
We have also  $v\in L^{\beta'}(I,L^{q'}(I'))$ and by \eqref{est2}, 
\begin{align*}
\|v\|_{L^{\beta'}(I,L^{q'}(I'))}&\le \mathbf{c} \|V_1z\|_{L^{\beta}(I,L^{q}(I'))}
\\
&\le \mathbf{c} \|V_1\|_{L^{2\eta_2}(I,L^{2\eta_1}(I'))}\|z\|_{L^2(Q)}
\\
&\le \mathbf{c} \|V\|_{L^{\eta_2}(I,L^{\eta_1}(I'))}^{\frac{1}{2}}\|z\|_{L^2(Q)}.
\end{align*}
In particular, we have
\[
\|v\|_{L^b(Q)}\le \mathbf{c}\|V\|_{L^{\eta_2}(I,L^{\eta_1}(I'))}^{\frac{1}{2}}\|z\|_{L^2(Q)},
\]
where we recall that $b:=\min(q',\beta')$.

Let $2\le a<b$. By interpolation, we obtain
\[
\|v\|_{L^a(Q)}\le \|v\|_{L^2(Q)}^{1-\theta}\|v\|_{L^b(Q)}^\theta,
\]
with $\theta=\frac{b(a-2)}{a(b-2)}$. Given \eqref{lim0}, we get
\[
\lim_{\tau \in \Lambda,\, |\tau| \rightarrow \infty}\|v\|_{L^a(Q)}=0.
\]
Now, $v\in L^{\beta'}(I,L^{q'}(I'))$ yields $u\in L^{\beta'}(I,L^{q'}(I'))$. We have $\frac{1}{q}=\frac{1}{\eta_1}+\frac{1}{q'}$ and $\frac{1}{\beta'}+\frac{1}{\eta_2}=\frac{1}{\tilde{\beta}}$, where
\begin{equation}\label{tildebeta}
\tilde{\beta}=\tilde{\beta}(\mu):=\frac{\mu q'-2}{2\mu -2}.
\end{equation}
We verify that $\mu \mapsto \tilde{\beta}(\mu)$ is a decreasing function from $\left(\frac{2}{4-q'},\infty\right)$ onto $ \left(\frac{q'}{2},2\right)$ and, since $\tilde{\beta}\mapsto \frac{2\tilde{\beta}}{2-\tilde{\beta}}$ is an increasing function from $ \left(\frac{q'}{2},2\right)$ onto $\left(\frac{2q'}{4-q'},\infty\right)\subset (2,\infty)$, we obtain
\begin{equation}\label{0.x.1}
\inf\left\{\frac{2\tilde{\beta}}{2-\tilde{\beta}};\; \mu \in \left(\frac{2}{4-q'},\infty\right)\right\}>2.
\end{equation} 
On the other hand, we can easily show that
\begin{equation}\label{0.x.2}
\inf\left\{\frac{2q}{2-q};\; q' \in \left(2,4\right)\right\}=4.
\end{equation}
Let 
\begin{equation}\label{hatbeta}
\hat{\beta}:=\min (\tilde{\beta},q). 
\end{equation}
Using \eqref{0.x.1} and \eqref{0.x.2}, we get
\begin{equation}\label{0.x.3}
\inf\left\{\frac{2\hat{\beta}}{2-\hat{\beta}};\; q' \in \left(2,4\right)\right\}>2.
\end{equation}

From \eqref{fh}, we have  $Vu\in L^{\tilde{\beta}}(I,L^q(I'))\subset L^{\hat{\beta}}(Q)$. As $\Delta u=Vu$, we have by the interior elliptic regularity that $u\in W^{2,\hat{\beta}}(\Omega)$. 
From $\hat{\beta}<2$, we obtain $W^{1,\hat{\beta}}(\Omega)\hookrightarrow L^{\frac{2\hat{\beta}}{2-\hat{\beta}}}(\Omega) \subset L^2(\Omega)$ by \eqref{0.x.3}. In particular, $u\in H^1(\Omega)$.

 In summary, the following result has been demonstrated.
 
\begin{theorem}\label{thmcgo}
Let $V\in L^{\eta_2}(I,L^{\eta_1}(I'))$. There exists $\tau_0(V)=\tau_0(\ell,\ell',q',\mu,\gamma,V)>0$ such that if $|\tau|\ge \tau_0(V)$ with $\tau \in \Lambda$ and $\xi \in \mathbb{C}$, we find $u\in \mathscr{S}_V$  of the form $u=e^{(\tau-\xi)(-x+ix')}+e^{-\tau x}v$, where $v\in H^1(\Omega)$ satisfies for all $2\le a<b:=\min(q',\beta')$
\begin{equation}\label{lim}
\lim_{\tau \in \Lambda,\, |\tau| \rightarrow \infty}\|v\|_{L^a(Q)}=0.
\end{equation} 
\end{theorem}

\section{Proof of Theorem \ref{thmu}}\label{s3}

The following useful lemma will be used in the proof of Theorem \ref{thmu}.
\begin{lemma}\label{lemii}
For $j=1,2$, let $V_j\in L_\ast^\alpha(\Omega )$ and $u_j\in \mathscr{S}_{V_j}$. Then the following integral identity holds:
\begin{equation}\label{ii}
\int_\Omega (V_2-V_1)u_1\overline{u_2}dx =\langle \left(\Sigma_{V_2}-\Sigma_{V_1}\right)(u_1{_{|\Gamma}}),u_2{_{|\Gamma}}\rangle ,
\end{equation}
where $\langle \cdot, \cdot \rangle $ is the duality pairing between $H^{\frac{1}{2}}(\Gamma)$ and $H^{-\frac{1}{2}}(\Gamma)$. 
\end{lemma}

We omit the proof of Lemma \ref{lemii}, as it is identical to that of \cite[Lemma 1.4]{Ch2}.

\begin{proof}[Proof of Theorem \ref{thmu}]
 Let $V_1,V_2\in \mathscr{V}_\mu$ satisfy $(V_2-V_1)\chi_\Omega=U^1\otimes U^2\in L^{\sigma_2}(I,\mathbb{R})\otimes L^{\sigma_1}(I', \mathbb{R})$ and $\Sigma_{V_1}=\Sigma_{V_2}$.

If $U^1=0$, there is nothing to prove. We may therefore assume that $U^1\not\equiv 0$.
 
Let $\tau_0=\max(\tau_0(V_1),\tau_0(V_2))$, where $\tau_0(V_j)$ is as  in Theorem \ref{thmcgo} when $V=V_j$, $j=1,2$, and $\tau \in \Lambda$ be such that $\tau \ge \tau_0$. According to Theorem \ref{thmcgo}, there exist $u_1\in \mathscr{S}_{V_1}$ and $u_2\in \mathscr{S}_{V_2}$ of the form
\begin{align*}
&u_1=e^{(\tau -\xi) (-x+ix')}+e^{-\tau x}v_1=e^{-\tau x}\left[e^{ -\xi (-x+ix')+i\tau x'}+v_1\right],
\\
&\overline{u_2}=e^{-\tau (-x+ix')}+e^{\tau x}v_2=e^{\tau x}(e^{-i\tau x'}+v_2) \quad(\tau\; \mbox{replaced by}\; -\tau,\; \xi=0),
\end{align*}
with
\[
 \lim_{\tau\in \Lambda,\, \tau \rightarrow +\infty}\|v_j\|_{L^a(Q)}=0,\quad j=1,2,
\] 
for all $2\le a<b:=\min(q',\beta')$. 

We verify that
\[
u_1\overline{u_2}=e^{-\xi (-x+ix')}+\tilde{u},
\]
where
\begin{equation}\label{uu}
\tilde{u}=e^{-i\tau x'}v_1+e^{ -\xi (-x+ix')+i\tau x'}v_2+v_1v_2.
\end{equation}

Let $V=\chi_\Omega(V_2-V_1)$. Using \eqref{ii}, we obtain
\begin{equation}\label{eqq2}
0=\int_\Omega Vu_1\overline{u_2}dx=\int_Q Vu_1\overline{u_2}dx=\widehat{U^1}(i\xi)\widehat{U^2}(\xi)+\int_Q V\tilde{u}dx.
\end{equation} 
 We now choose $2\le a<b$ such that
 \[
\frac{1}{b'}=\frac{1}{a}+\frac{1}{\sigma},
 \]
 where $b'=\frac{b}{b-1}$ is the conjugate of $b$. By applying  \eqref{fh}, we get
 \begin{align*}
 \left|\int_Q V\tilde{u}dx \right|&\le \|e^{-i\tau x'}\|_{L^{a'}(Q)}\|v_1\|_{L^a(Q)}
 \\
&\hskip .5cm+\|e^{ -\xi (-x+ix')+i\tau x'}\|_{L^{a'}(Q)}\|v_2\|_{L^a(Q)} +\|Vv_1\|_{L^{b'}(Q)}\|v_2\|_{L^b(Q)}
 \\
 &\le \|e^{-i\tau x'}\|_{L^{a'}(Q)}\|v_1\|_{L^a(Q)}
 \\
&\hskip .3cm+\|e^{ -\xi (-x+ix')+i\tau x'}\|_{L^{a'}(Q)}\|v_2\|_{L^a(Q)} +\|V\|_{L^{\sigma}(Q)}\|v_1\|_{L^a(Q)}\|v_2\|_{L^b(Q)}.
 \end{align*}
 Hence,
\[
\int_\Omega V\tilde{u}dx \rightarrow 0\quad \mathrm{as}\; \tau \in \Lambda,\; \tau \rightarrow +\infty.
\]
By passing to the limit, as $\tau \in \Lambda$, $\tau \rightarrow +\infty$, in \eqref{eqq2}, we obtain $\widehat{U^1}(i\xi)\widehat{U^2}(\xi)=0$ for all $\xi \in \mathbb{C}$. As $\widehat{U^1}$ is a non-zero entire function, its zeros are isolated points. Consequently, $\widehat{U^2}=0$ and hence $U^2=0$.
\end{proof}

\section{Stability inequality}\label{s4}

In order to establish a stability inequality, we need to quantify \eqref{lim}. For this, we reduce the class of potentials. In the following, the constants $\vartheta >0$, $t_0>0$, $\mathfrak{c}>0$ and $\varkappa>0$ are given. 

Let  $V=|V|e^{i\varphi}\in L^{\eta_2}(I,L^{\eta_1}(I'))$ with $\varphi:=\arg(V)$ satisfying 
\begin{equation}\label{cos}
\||V|^{\frac{1}{2}}\chi_{\{|V|^{\frac{1}{2}}>t\}}\|_{L^{2\eta_2}(I,L^{2\eta_1}(I'))}\le \mathfrak{c}t^{-\vartheta},\quad t\ge t_0.
\end{equation}
If $\tau \in \Lambda$, $|\tau|\ge t_0^4$, $V_0=|V|^{\frac{1}{2}}e^{i\varphi}$ and $V_1=|V|^{\frac{1}{2}}$, we set
\[
V_j^0:=V_j\chi_{\{|V_j|\le |\tau|^{\frac{1}{4}}\}}
\quad\text{and}\quad V_j^1:=V_j\chi_{\{|V_j|> |\tau|^{\frac{1}{4}}\}},\quad j=0,1.
\]
We verify
\begin{align*}
&\|V_1^0\|_{L^\infty (Q)}\le |\tau|^{\frac{1}{4}}\quad\text{and}\quad \|V_j^0\|_{L^{2\eta_2}(I,L^{2\eta_1}(I'))}\le \|V\|_{L^{\eta_2}(I,L^{\eta_1}(I'))}^{\frac{1}{2}},\quad j=0,1,
\\
&\|V_j^1\|_{L^{2\eta_2}(I,L^{2\eta_1}(I'))}\le \mathfrak{c}|\tau|^{-\frac{\vartheta}{4}},\quad j=0,1.
\end{align*}

Let $\aleph_0=(\Omega, \ell,\ell',q',\mu,\varkappa,t_0,\vartheta,\mathfrak{c})$. From now on, $\mathbf{c}=\mathbf{c}(\aleph_0)>0$ will denote a generic constant.

Assume that $\|V\|_{L^{\eta_2}(I,L^{\eta_1}(I'))}\le \varkappa$. Let $\xi \in \mathbb{R}$ and $z$ be as in \eqref{zz} and $v=E_\tau V_1z$. We proceed as above to demonstrate 
\[
\|v\|_{L^2(Q)}\le \mathbf{c}(|\tau|^{-\frac{3}{4}}+|\tau|^{-\frac{\vartheta}{4}})e^{\mathbf{c}|\xi|}.
\]
Thus,
\begin{equation}\label{wy0}
\|v\|_{L^2(Q)}\le \mathbf{c}|\tau|^{-\vartheta^\ast}e^{\mathbf{c}|\xi|},
\end{equation}
where $\vartheta^\ast=\frac{1}{4}\min (3,\vartheta)$. Also, we have
\begin{align*}
\|v\|_{L^{\beta'}(I,L^{q'}(I'))}&\le \mathbf{c}\|V_1z\|_{L^\beta(I,L^q(I'))}
\\
&\le  \mathbf{c} \|V_1\|_{L^{2\eta_2}(I,L^{2\eta_1}(I'))}\|z\|_{L^2(Q)}
\\
&\le \mathbf{c}\|e^{\xi x+i\zeta x'}\|_{L^2(Q)}.
\end{align*}
That is, we have
\[
\|v\|_{L^{\beta'}(I,L^{q'}(I'))}\le \mathbf{c}e^{\mathbf{c}|\xi|}
\]
and then
\begin{equation}\label{wy0.0}
\|v\|_{L^b(Q)}\le \mathbf{c}e^{\mathbf{c}|\xi|},
\end{equation}
where we recall that $b=\min(q',\beta')$.
Let $2\le a< b$ be such that
 \[
\frac{1}{b'}=\frac{1}{a}+\frac{1}{\sigma}.
 \]
Using \eqref{wy0}, \eqref{wy0.0} and an interpolation inequality, we obtain
\begin{equation}\label{wy0.1}
\|v\|_{L^a(Q)}\le \mathbf{c}|\tau|^{-\tilde{\vartheta}}e^{\mathbf{c}{\xi}}.
\end{equation}
Here, $\tilde{\vartheta}=\frac{2(b-a)}{a(b-2)}\vartheta^\ast$.

If $u=\phi +e^{-\tau x}v$, where $\phi$ is as in \eqref{phi}, then
\begin{equation}\label{wy1}
\|u\|_{L^{\beta'}(I,L^{q'}(I'))}\le e^{\mathbf{c}(|\tau|+|\xi|)}.
\end{equation}

On the other hand, by $W^{2,\hat{\beta}}$-a priori estimate (e.g., \cite[Theorem 9.11]{GT}), we get
\[
\|u\|_{W^{2,\hat{\beta}}(\Omega)}\le c\left(\|Vu\|_{L^{\hat{\beta}}(Q)}+\|u\|_{L^{\hat{\beta}}(Q)}\right),
\]
where $\hat{\beta}$ is given by \eqref{hatbeta} and $c=c(\Omega,Q)>0$ is a constant. Hence, \eqref{wy1} yields
\[
\|u\|_{W^{2,\hat{\beta}}(\Omega)}\le e^{\mathbf{c}(|\tau|+|\xi|)},
\]
which, combined with the fact that $W^{2,\hat{\beta}}(\Omega)$ is continuously embedded in $H^1(\Omega)$, gives
\begin{equation}\label{wy2}
\|u\|_{H^1(\Omega)}\le e^{\mathbf{c}(|\tau|+|\xi|)}.
\end{equation}

Let $\sigma$ be as in the statement of Theorem \ref{thmu}, and denote by $\mathscr{V}_\mu ^\ast$ the set of functions $V\in \mathscr{V}_\mu$ satisfying \eqref{cos} and $\|V\|_{L^{\eta_2}(I,L^{\eta_1}(I'))}\le \varkappa$. Let $V_1,V_2\in \mathscr{V}_\mu ^\ast$ such that $\chi_\Omega(V_1-V_2)=U^1\otimes U^2$, where $U^1\in L^\sigma(I, \mathbb{R})$ satisfies $\pm U^1\ge 0$ and $\|U^1\|_{L^1(I)}\ne 0$ and $U^2\in L^\sigma(\mathbb{R}, \mathbb{R})\cap H^\varsigma (\mathbb{R})$ for some $\varsigma >0$ satisfies $\|U^2\|_{H^\varsigma (\mathbb{R})}\le \varkappa$.

Let $\tilde{u}$ be given by \eqref{uu} corresponding to $V_1,V_2\in \mathscr{V}_\mu ^\ast$. From the proof of Theorem \ref{thmu}, we have
\[
\left|\int_Q V\tilde{u}dx \right|\le \mathbf{c}(1+\|V\|_{L^{\sigma}(Q)})\left(\|v_1\|_{L^a(Q)}+\|v_2\|_{L^a(Q)} +\|v_1\|_{L^a(Q)}\|v_2\|_{L^b(Q)}\right),
\]
which, in combination with \eqref{wy0.0} and \eqref{wy0.1}, implies
\[
\|\tilde{u}\|_{L^2(Q)}\le \mathbf{c}(1+\|V\|_{L^{\sigma}(Q)})|\tau|^{-\tilde{\vartheta}}e^{\mathbf{c}|\xi|}.
\]
This and \eqref{ii} yield
\[
|\widehat{U^1}(i\xi)\widehat{U^2}(\xi)|\le \mathbf{c}(1+\|V\|_{L^{\sigma}(Q)})|\tau|^{-\tilde{\vartheta}}e^{\mathbf{c}|\xi|}+|\langle \left(\Sigma_{V_2}-\Sigma_{V_1}\right)(u_1{_{|\Gamma}}),u_2{_{|\Gamma}}\rangle|.
\]

In the remaining part of this section, $\mathbf{c}_0=\mathbf{c}_0(\aleph)>0$ is a generic constant, where $\aleph=(\Omega, \ell,\ell',q',\mu,\beta,\varkappa,\varsigma,U^1,t_0,\vartheta,\mathfrak{c})$.

As $\pm U^1\ge 0$ and $\|U^1\|_{L^1(I)}\ne 0$, the inequality above yields
\begin{align*}
\mathbf{c}_0|\widehat{U^2}(\xi)|
&\le |\tau|^{-\tilde{\vartheta}}e^{\mathbf{c}|\xi|}+|\langle \left(\Sigma_{V_2}-\Sigma_{V_1}\right)(u_1{_{|\Gamma}}),u_2{_{|\Gamma}}\rangle|
\\
&\le |\tau|^{-\tilde{\vartheta}}e^{\mathbf{c}|\xi|}+\|\Sigma_{V_2}-\Sigma_{V_1}\|\|u_1\|_{H^{\frac{1}{2}}(\Gamma)}\|u_2\|_{H^{\frac{1}{2}}(\Gamma)}.
\end{align*}
From \eqref{wy2}, we get
\[
\|u_j\|_{H^{\frac{1}{2}}(\Gamma)}\le e^{\mathbf{c}(|\tau| +|\xi|)},\quad j=1,2.
\]
Hence,
\[
\mathbf{c}_0|\widehat{U^2}(\xi)|\le |\tau|^{-\tilde{\vartheta}}e^{\mathbf{c}|\xi|}+e^{\mathbf{c}(|\tau| +|\xi|)}\|\Sigma_{V_2}-\Sigma_{V_1}\|.
\]
Therefore, for all $\rho >0$ and $\tau \in \Lambda\cap [t_0^4,\infty)$, we obtain
\[
\mathbf{c}_0\|\widehat{U^2}\|_{L^2(\{|\xi|\le \rho\})}\le  \tau ^{-\tilde{\vartheta}}e^{\mathbf{c}\rho}+e^{\mathbf{c}(\tau +\rho)}\|\Sigma_{V_2}-\Sigma_{V_1}\|.
\]
Since $\mathbf{c}_0$, $\mathbf{c}$ and $\vartheta^\ast$ do not depend on $\tau$, we observe that the preceding inequality holds for all $\tau \in [\tilde{\tau}, \infty)$, where $\tilde{\tau} = \max(4, t_0^4)$.

We choose $\rho$ in such a way that $e^{\mathbf{c}\rho}= \tau ^{\frac{\tilde{\vartheta}}{2}}$, that is, $\rho=\mathbf{c}\ln \tau$. With this choice of $\rho$, the inequality above gives
\begin{equation}\label{wy3}
\mathbf{c}_0\|\widehat{U^2}\|_{L^2(\{|\xi|\le \mathbf{c}\ln \tau\})}\le \tau ^{-\frac{\tilde{\vartheta}}{2}}+e^{\mathbf{c}\tau }\|\Sigma_{V_2}-\Sigma_{V_1}\|.
\end{equation}
On the other hand, we verify that
\[
\|\widehat{U^2}\|_{L^2(\{|\xi|> \mathbf{c}\ln \tau\})}\le (\ln \tau)^{-\varsigma}\|U^2\|_{H^\varsigma (\mathbb{R})},
\]
which, in combination with \eqref{wy3}, yields
\[
\mathbf{c}_0\|U^2\|_{L^2(\mathbb{R})}=\mathbf{c}_0\|\widehat{U^2}\|_{L^2(\mathbb{R})}\le \tau ^{-\frac{\tilde{\vartheta}}{2}}+(\ln \tau)^{-\varsigma}+e^{\mathbf{c}\tau }\|\Sigma_{V_2}-\Sigma_{V_1}\|.
\]
Since the function $\tau\in [\tilde{\tau},\infty)\mapsto \tau ^{-\frac{\tilde{\vartheta}}{2}}(\ln \tau)^\varsigma$ is bounded, there exists $\tau_\ast\ge \tilde{\tau}$ such that, for all $\tau \in [\tau_\ast,\infty)$, we have
\begin{equation}\label{wy4}
\mathbf{c}_0\|U^2\|_{L^2(\mathbb{R})}\le (\ln \tau)^{-\varsigma}+e^{\mathbf{c}\tau }\|\Sigma_{V_2}-\Sigma_{V_1}\|.
\end{equation}

Define $H: [\tau_\ast,\infty[ \rightarrow ]0, \infty[: \tau \mapsto H(\tau)=(\ln \tau)^{-\varsigma}e^{-\mathbf{c}\tau}$, where $\mathbf{c}$ is the constant in \eqref{wy4}. We verify that $H$ is a decreasing bijection from $[\tau_\ast,\infty[$ onto $]0,a_\ast]$, where $a_\ast=H(\tau_\ast)$. If $\|\Sigma_{V_2}-\Sigma_{V_1}\|\le a_\ast$, then there exists $\tau\in [\tau_\ast,\infty[$ such that $H(\tau)=\|\Sigma_{V_2}-\Sigma_{V_1}\|$. That is, we have $(\ln \tau)^{-\varsigma}e^{-\mathbf{c}\tau}=\|\Sigma_{V_2}-\Sigma_{V_1}\|$, and hence 
\[
\ln \frac{1}{\|\Sigma_{V_2}-\Sigma_{V_1}\|^{\frac{1}{\mathbf{c}+\varsigma}}}\le \tau
\]

In addition, if  $\|\Sigma_{V_2}-\Sigma_{V_1}\|< e^{-(\mathbf{c}+\varsigma)}$, then 
\[
\ln \ln \frac{1}{\|\Sigma_{V_2}-\Sigma_{V_1}\|^{\frac{1}{\mathbf{c}+\varsigma}}}\le \ln \tau.
\]
Consequently, \eqref{wy4} implies
\begin{equation}\label{wy5}
\|U^2\|_{L^2(\mathbb{R})}\le \mathbf{c}_0\left(\ln \ln \frac{1}{\|\Sigma_{V_2}-\Sigma_{V_1}\|^{\frac{1}{\mathbf{c}+\varsigma}}}\right)^{-\varsigma}.
\end{equation}
Let $b_\ast:=\min(a_\ast, e^{-(\mathbf{c}+\varsigma)})$. In this case, when $\|\Sigma_{V_2}-\Sigma_{V_1}\|\ge  b_\ast$, we have trivially 
\begin{equation}\label{wy6}
\|U^2\|_{L^2(\mathbb{R})}\le \mathbf{c}_0\|\Sigma_{V_2}-\Sigma_{V_1}\|.
\end{equation}

Define
\begin{equation}\label{Psi}
\Psi(\rho):=\chi_{[b_\ast,\infty[}(\rho)\rho+ \chi_{]0, b_\ast[}(\rho)(\ln |\ln \rho|)^{-\varsigma}.
\end{equation}

In view of \eqref{wy5} and \eqref{wy6}, we have
\[
\|U^2\|_{L^2(\mathbb{R})}\le \mathbf{c}_0 \Psi(\|\Sigma_{V_2}-\Sigma_{V_1}\|).
\]

In summary, we have established the following result.

\begin{theorem}\label{thmst}
Let $\beta':=\frac{\mu q'-2}{\mu(q'-2)}$ $(>2)$, $b:=\min(q',\beta')$ and
\[
 \frac{b}{b-2}<\sigma <\frac{2b}{b-2}.
\]
Let $V_1,V_2\in \mathscr{V}_\mu ^\ast$ such that $\chi_\Omega(V_1-V_2)=U^1\otimes U^2$, where $U^1\in L^\sigma(I, \mathbb{R})$ satisfies $\pm U^1\ge 0$ and $\|U^1\|_{L^1(I)}\ne 0$ and $U^2\in L^\sigma(\mathbb{R}, \mathbb{R})\cap H^\varsigma (\mathbb{R})$ for some $\varsigma >0$ satisfies $\|U^2\|_{H^\varsigma (\mathbb{R})}\le \varkappa$. Then we have
\[
\|U^2\|_{L^2(\mathbb{R})}\le \mathbf{c}_0 \Psi(\|\Sigma_{V_2}-\Sigma_{V_1}\|),
\]
where $\Psi$ is given by \eqref{Psi}.
\end{theorem}

\appendix
\section{Proof of Theorem \ref{thmci}}\label{appendixA}

We first prove a lemma that will be used to establish Theorem \ref{thmci}. To this end, recall that $I'=]a',b'[$ and $\ell'=b'-a'$. Define the operator $A'$ as follows
\[
D(A')=\{w\in H_0^1(I'); \; w''\in L^2(I')\},\quad A'w=-w''.
\]
Let $(\lambda_j^2)_{j\ge 1}$ be the nondecreasing sequence of eigenvalues of the operator $A'$ and recall that
\[
\lambda_j=\frac{j\pi}{\ell'},\quad j\in \mathbb{N}.
\]
Let $(\phi_j)_{j\in \mathbb{N}}$ be an orthonormal basis of $L^2(I')$ consisting of eigenfunctions such that $A'\phi_j=\lambda_j^2 \phi_j$, for all $j\in \mathbb{N}$. For all $\lambda \ge 0$, let $\pi_\lambda:L^2(I') \rightarrow L^2(I')$ be given as follows
\[
\pi_\lambda w:=\sum_{\lambda_j\in [\lambda,\lambda+1)}(w|\phi_j)\phi_j,\quad w\in L^2(I'),
\]
if $\{j\in \mathbb{N};\; \lambda_j\in [\lambda,\lambda+1)\}\ne \emptyset $, and $\pi_\lambda=0$ if $\{j\in \mathbb{N};\; \lambda_j\in [\lambda,\lambda+1)\}= \emptyset $. 

In the remaining part of this appendix, $\mathbf{c}=\mathbf{c}(\ell',q',\mu)>0$ will denote a generic constant.

\begin{lemma}\label{lem1}
For all $\lambda\ge 0$ and $w\in L^2(I')$ we have
\begin{align}
&\|\pi_\lambda w\|_{L^{q'}(I')}\le \mathbf{c}(1+\lambda)^{\theta}\|w\|_{L^2(I')}  ,\label{up4}
\\
&\|\pi_\lambda w\|_{L^2(I')}\le \mathbf{c}(1+\lambda)^{\theta}\|w\|_{L^q(I')},\label{up5}
\end{align}
where $\theta$ is given by \eqref{theta}.
\end{lemma}

\begin{proof}
Let $p'=\mu q'$, $\lambda\ge 0$ and $w\in L^2(I')$. Using 
\[
\|\pi_\lambda w\|_{L^2(I')}\le \|w\|_{L^2(I')},\quad \|\nabla \pi_\lambda w\|_{L^2(I')}\le (1+\lambda) \|w\|_{L^2(I')},
\]
$\pi_\lambda w\in H_0^1(I')\hookrightarrow L^\infty (I')$ and an elementary interpolation inequality, we obtain
\begin{align*}
\|\pi_\lambda w\|_{L^{q'}(I')} &\le \|\pi_\lambda w\|_{L^2(I')}^{1-\theta}\|\pi_\lambda w\|_{L^{p'}(I')}^\theta
\\
&\le \mathbf{c} \|\pi_\lambda w\|_{L^2(I')}^{1-\theta}\|\nabla \pi_\lambda w\|_{L^2(I')}^\theta
\\
&\le \mathbf{c}(1+\lambda)^\theta \|\pi_\lambda w\|_{L^2(I')}.
\end{align*}
In other words, we proved \eqref{up4}. As $\pi_\lambda$ is self-adjoint and $\pi_\lambda ^2=\pi_\lambda$, we have
 \[
\|\pi_\lambda w\|_{L^2(I')}^2=(\pi_\lambda w|\pi_\lambda w)=(\pi_\lambda^2w|w)=(\pi_\lambda w|w), 
\]
which, in combination with H\"older's inequality, gives
 \[
\|\pi_\lambda w\|_{L^2(I')}^2\le \|\pi_\lambda w\|_{L^{q'}(I')} \|\pi_\lambda w\|_{L^q(I')}. 
\]
This and \eqref{up4} imply \eqref{up5}.
\end{proof}

\begin{proof}[Proof of Theorem \ref{thmci}]

Let $|\tau|\in [4,\infty )\setminus \{\lambda_j;\; j\in \mathbb{N}\}$, $u\in C_0^\infty(\mathbb{R}^2)$ with $\mathrm{supp}(u)\subset Q$. The following identities enable us to reduce the proof to the case $\tau \ge 4$: $\|u\|_{L^\gamma(\mathbb{R},L^{q'}(I'))}=\|u(-\cdot,\cdot)\|_{L^\gamma(\mathbb{R},L^{q'}(I'))}$ and
\[
\|(e^{-\tau t}\Delta e^{\tau t})u\|_{L^\beta(\mathbb{R}, L^q(I'))}=\|(e^{\tau t}\Delta e^{-\tau t})u(-\cdot,\cdot)\|_{L^\beta(\mathbb{R}, L^q(I'))}.
\] 
Set
\[
f:=e^{\tau t} (\partial_t^2+\partial_{x'}^2) e^{-\tau t}u=(\partial_t^2-2\tau \partial_t +\tau^2 +\partial_{x'}^2)u,
\]
and let $\mathcal{F}$ denote the Fourier transform with respect to $t$, and $\xi \in \mathbb{R}$.  Let 
\[
p_j: L^{q'}(I')\rightarrow L^{q'}(I'): w \mapsto (w|\phi_j)\phi_j.
\]
Since 
\[
p_j(\partial_{x'}^2u)=-\lambda_j^2 p_j u, 
\]
we obtain
\begin{align*}
\mathcal{F}(p_jf(\cdot,x'))(\xi)&= ((i\xi)^2-2i\tau \xi +\tau^2-\lambda_j^2)\mathcal{F}(p_ju(\cdot,x'))(\xi)
\\
&=((i\xi-\tau)^2-\lambda_j^2)\mathcal{F}(p_ju(\cdot,x'))(\xi)
\\
&=(i\xi-\tau-\lambda_j)(i\xi-\tau+\lambda_j)\mathcal{F}(p_ju(\cdot,x'))(\xi).
\end{align*}

In the following, $t\in \mathbb{R}$ and $x'\in I'$. Since $\tau \ne \lambda_j$ for all $j\ge 1$, we obtain
\begin{align*}
p_ju (t,x')&=\frac{1}{2\pi}  \int_\mathbb{R}\frac{e^{it\xi}}{(i\xi-\tau-\lambda_j)(i\xi-\tau+\lambda_j)}\mathcal{F}(p_jf(\cdot,x'))(\xi)d\xi
\\
&= \frac{1}{2\pi}  \int_{\mathbb{R}}\int_\mathbb{R}\frac{e^{i(t-s)\xi}}{(i\xi-\tau-\lambda_j)(i\xi-\tau+\lambda_j)}p_jf(s,x'))d\xi ds.
\end{align*}
Hence,
\[
p_ju (t,x')=\int_{\mathbb{R}}m_j^\tau(t-s)p_jf(s,x')ds,
\]
where
\[
m_j^\tau(\eta):=\frac{1}{2\pi}  \int_\mathbb{R}\frac{e^{i\eta\xi}}{(i\xi-\tau-\lambda_j)(i\xi-\tau+\lambda_j)}d\xi.
\]
The following inequality
\[
\|u(t,\cdot)\|_{L^{q'}(I')}=\left\|\sum_{k\ge 0}\pi_k^2u(t,\cdot)\right\|_{L^{q'}(I')}\le \sum_{k\ge 0}\|\pi_k^2u(t,\cdot)\|_{L^{q'}(I')},
\]
combined with \eqref{up4}, implies
\begin{equation}\label{eq1}
\|u(t,\cdot)\|_{L^{q'}(I')}\le \mathbf{c}\sum_{k\ge 0}(1+k)^{\theta} \|\pi_k u(t,\cdot)\|_{L^2(I')}.
\end{equation}

On the other hand, we have
\begin{align*}
\|\pi_k u(t,\cdot)\|_{L^2(I')}^2&=\sum_{k\le \lambda_j<k+1}|(u(t,\cdot)|\phi_j)|^2
\\
&=\sum_{k\le \lambda_j<k+1}\left|\int_{\mathbb{R}} m_j^\tau(t-s)(f(s,\cdot)|\phi_j)ds \right|^2.
\end{align*}
Applying Minkowski's inequality, we obtain
\begin{align*}
\|\pi_k u(t,\cdot)\|_{L^2(I')}&\le \int_{\mathbb{R}}\left(\sum_{k\le \lambda_j<k+1}|m_j^\tau(t-s)(f(s,\cdot)|\phi_j)|^2\right)^{\frac{1}{2}}ds
\\
&\le \int_{\mathbb{R}}\max_{k\le \lambda_j<k+1}|m_j^\tau(t-s)|\left(\sum_{k\le \lambda_j<k+1}|(f(s,\cdot)|\phi_j)|^2\right)^{\frac{1}{2}}ds
\\
&\le \int_{\mathbb{R}}\max_{k\le \lambda_j<k+1}|m_j^\tau(t-s)|\|\pi_kf(s,\cdot)\|_{L^2(I')}ds,
\end{align*}
which, in light of \eqref{up5}, gives
\[
\|\pi_k u(t,\cdot)\|_{L^2(I')}\le \mathbf{c}\int_{\mathbb{R}}(1+k)^\theta\max_{k\le \lambda_j<k+1}|m_j^\tau(t-s)|\|f(s,\cdot)\|_{L^p(I')}ds.
\]
This  in \eqref{eq1} yields
\[
\|u(t,\cdot)\|_{L^{q'}(I')}\le \mathbf{c}\sum_{k\ge 0}(1+k)^{2\theta} \int_{\mathbb{R}}\max_{k\le \lambda_j<k+1}|m_j^\tau(t-s)|\|f(s,\cdot)\|_{L^{p}(I')}ds,
\]
which we rewrite as
\begin{equation}\label{eq2}
\|u(t,\cdot)\|_{L^{q'}(I')}\le \mathbf{c}\sum_{k\ge 0}A_k(t),
\end{equation}
where
\[
A_k(t):=(1+k)^{2\theta} \int_{\mathbb{R}}\max_{k\le \lambda_j<k+1}|m_j^\tau(t-s)|\|f(s,\cdot)\|_{L^q(I')}ds,\quad k\ge 0.
\]

It follows from \cite[Lemma 2.3]{DKS} that 
\[
|m_j^\tau(t-s)|\le \frac{1}{\lambda_j}e^{-|\tau -\lambda_j||t-s|},\quad j\ge 1.
\]

For $1\le k\le \lfloor \tau\rfloor-2$, we have
\[
\max_{k\le \lambda_j<k+1}|m_j^\tau(t-s)|\le \frac{1}{k}\max_{k\le \lambda_j<k+1}e^{-(\tau -\lambda_j)|t-s|}\le \frac{e^{-(\tau-1-k)|t-s|}}{k}.
\]
Hence,
\begin{align*}
\sum_{k=2}^{\lfloor \tau\rfloor-2}A_k(t) &\le  \sum_{k=2}^{\lfloor \tau\rfloor-2} (1+k)^{2\theta}k^{-1}\int_{\mathbb{R}}e^{-(\tau-1-k)|t-s|}\|f(s,\cdot)\|_{L^q(I')}ds
\\
&\le 2^{2\theta}\sum_{k=2}^{\lfloor \tau\rfloor-2} k^{2\theta-1}\int_{\mathbb{R}}e^{-(\tau-1-k)|t-s|}\|f(s,\cdot)\|_{L^q(I')}ds.
\end{align*}

As $2\theta-1<0$, we proceed as in \cite{DKS} to obtain
\[
\sum_{k=2}^{\lfloor \tau\rfloor-2}A_k(t)\le  2^{2\theta}\int_{\mathbb{R}}\|f(s,\cdot)\|_{L^q(I')}\int_1^{\lfloor \tau\rfloor-2}r^{2\theta-1}e^{-(\lfloor \tau\rfloor-2-r)|t-s|}drds.
\]
Then a change of variable yields
\begin{equation}\label{eq5}
\sum_{k=2}^{\lfloor \tau\rfloor-2} A_k(t)
 \le \mathbf{c}\int_{\mathbb{R}}|t-s|^{-2\theta}\|f(s,\cdot)\|_{L^q(I')}ds.
\end{equation}

Next, we discuss the case $k\ge \lfloor \tau\rfloor+2$, for which we have
\[
\max_{k\le \lambda_j<k+1}|m_j^\tau(t-s)|\le \frac{1}{k}\max_{k\le \lambda_j<k+1}e^{-(\lambda_j-\tau)|t-s|}\le\frac{e^{-(k-\tau)|t-s|}}{k}.
\]
Define
\[
h(\rho)=\rho^{2\theta-1}e^{-(\rho-\tau)|t-s|},\quad \rho\ge \lfloor \tau\rfloor+2.
\]
Since
\[
h'(\rho)=\left[-\alpha -|t-s|\rho\right] \rho^{2\theta-2} e^{-(\rho-\tau)|t-s|}\le 0,\quad \rho\ge \lfloor \tau\rfloor+1,
\]
we get
\begin{equation}\label{eq6}
\sum_{k\ge  \lfloor \tau\rfloor+2}A_k(t)
  \le \mathbf{c} \int_{\mathbb{R}}J_\tau(|t-s|)\|f(s,\cdot)\|_{L^q(I')}ds,
\end{equation}
where
\[
J_\tau(|t-s|)=\int_{ \lfloor \tau\rfloor+1}^\infty \rho^{2\theta-1}e^{-(\rho-\tau)|t-s|}d\rho.
\]
If $(\lfloor \tau\rfloor+1)|t-s|<1$, then 
\begin{align*}
J_\tau(|t-s|)&\le |t-s|^{-2\theta}\int_{ (\lfloor \tau\rfloor+1)|t-s|}^\infty r^{2\theta-1}e^{-(r-\tau |t-s|)}dr
\\
&\le |t-s|^{-2\theta}\left(\int_0^1r^{2\theta-1}d\rho+\int_1^\infty e^{-r+1}dr\right)
\\
&\le \mathbf{c} |t-s|^{-2\theta}.
\end{align*}
When $(\lfloor \tau\rfloor+1)|t-s|\ge 1$, we have
\begin{align*}
J_\tau(|t-s|)&\le |t-s|^{-2\theta}\int_{ (\lfloor \tau\rfloor+1)|t-s|}^\infty e^{-(r-\tau |t-s|)}dr
\\
&=  |t-s|^{-2\theta} \int_{(\lfloor \tau\rfloor+1-\tau)|t-s|}^\infty e^{-r}dr
\\
&\le  |t-s|^{-2\theta}\int_0^\infty e^{-r}dr
\\
&\le \mathbf{c} |t-s|^{-2\theta}.
\end{align*}
The preceding inequalities in \eqref{eq6} yield
\begin{equation}\label{eq7}
\sum_{k\ge  \lfloor \tau\rfloor+2}A_k(t) 
 \le \mathbf{c} \int_{\mathbb{R}}|t-s|^{-2\theta}\|f(s,\cdot)\|_{L^q(I')}ds.
\end{equation}

When $k=0$, we have from \cite[Lemma 2.3]{DKS} 
\[
\max_{0\le \lambda_j<1}|m_j^\tau (t-s)|\le e^{-\frac{\tau}{2}|t-s|}\le e^{-2|t-s|}.
\]
As $\sup_{\eta >0}\eta^{2\theta}e^{-2\eta}<\infty$, we obtain
\begin{equation}\label{eq11}
A_0(t)\le \int_{\mathbb{R}}|t-s|^{-2\theta}\|f(s,\cdot)\|_{L^q(I')}ds.
\end{equation}
While for $k=1$, since
\[
\max_{1\le \lambda_j<2}|m_j^\tau (t-s)|\le e^{-(\tau-2)|t-s|}\le e^{-(\tau-2)|t-s|}\le e^{-2|t-s|} ,
\]
proceeding similarly as for $A_0(t)$, we get
\begin{equation}\label{eq11.1}
A_1(t)\le \int_{\mathbb{R}}|t-s|^{-2\theta}\|f(\cdot,s)\|_{L^q(I')}ds.
\end{equation}

Let us now consider the remaining case $k\in \{\lfloor \tau\rfloor-1,\lfloor \tau\rfloor, \lfloor \tau\rfloor+1\}$, for which we have
\[
\max_{k\le \lambda_j<k+1}|m_j^\tau(t-s)|\le 1.
\]
Therefore, we have
\begin{equation}\label{eq12}
A_k(t)\le \int_{\mathbb{R}}\|f(s,\cdot)\|_{L^q(I')}ds.
\end{equation}

Putting together \eqref{eq2}, \eqref{eq5}, \eqref{eq7}, \eqref{eq11}, \eqref{eq11.1} and \eqref{eq12}, we end up getting
\begin{equation}\label{eq8}
\|u(t,\cdot)\|_{L^{q'}(I')}\le \mathbf{c}\left( \int_{\mathbb{R}}|t-s|^{-2\theta}\|f(s,\cdot)\|_{L^q(I')}ds+\int_{\mathbb{R}}\|f(s,\cdot)\|_{L^q(I')}ds\right).
\end{equation}

Finally, applying Hardy-Littlewood-Sobolev's inequality to the first term of the right hand side of \eqref{eq8}, we get
\[
\|u\|_{L^\gamma(\mathbb{R},L^{q'}(I'))}\le \mathbf{c}(1+\ell^{\frac{1}{\beta'}+\frac{1}{\gamma}})\|(e^{\tau t}\Delta e^{-\tau t})u\|_{L^\beta(\mathbb{R},L^q(I'))}.
\]
We complete the proof by noting that the preceding inequality remains true by continuity for all $\tau \ge 4$.
\end{proof}

\end{document}